\documentclass{svjour3}

\smartqed

\usepackage{amsmath,amssymb,bm}
\usepackage{graphicx}
\usepackage{booktabs}
\usepackage{multirow}
\usepackage{color}
\usepackage{mathrsfs}
\usepackage{dsfont}
\usepackage{cases}
\usepackage{epstopdf}
\usepackage{enumitem}
\usepackage{mathtools}

\allowdisplaybreaks

\newtheorem{thm}{Theorem}
\newtheorem{prop}{Proposition}

\newtheorem{rem}{Remark}
\newtheorem{exa}{Example}
\numberwithin{equation}{section}

\begin{document}

\title{Positivity-preserving scalar auxiliary variable schemes for gradient flows via a quadratic reformulation}

\titlerunning{Positivity-preserving SAV schemes}

\author{Qiong-Ao Huang \and Zhi-Hao Liu \and Ying-Wei Wang \and Li-Na Yan}

\authorrunning{Q.-A. Huang et al.}

\institute{Qiong-Ao Huang \at
  School of Mathematics and Statistics, Henan University, Kaifeng 475004, China \\
  Center for Applied Mathematics of Henan Province, Henan University, Zhengzhou 450046, China \\
  \email{huangqiongao@henu.edu.cn}
\and
  Zhi-Hao Liu \and Ying-Wei Wang \and Li-Na Yan \at
  School of Mathematics and Statistics, Henan University, Kaifeng 475004, China \\
  \email{zhihaoliu@henu.edu.cn; wangyingwei@henu.edu.cn; yanlina@henu.edu.cn}
}

\date{Received: date / Accepted: date}

\maketitle

\begin{abstract}
The scalar auxiliary variable (SAV) method replaces the nonlinear part of the free energy
by a positive scalar $r(t)=\sqrt{\mathcal{E}_{_\mathcal{N}}[\phi]+C}>0$,
thereby yielding linear, unconditionally energy-stable schemes for gradient flows.
At the discrete level, however, the standard backward Euler and Crank--Nicolson discretizations
provide no guarantee that the computed $r^{n+1}$ remains positive,
an inconsistency with the continuous definition that contradicts the
square-root ansatz and may compromise long-time robustness.
Although the SAV method has been widely applied,
this subtle but consequential issue has received little attention.
We first characterize this failure quantitatively by deriving a sharp
criterion and a sufficient condition on the time step size, and construct an
explicit counterexample showing that sign loss occurs for parameters of practical relevance.
Rather than modifying the definition of $r$ as in existing positivity-preserving variants,
we retain the square-root form and reformulate the discrete evolution from $r_t$ to $(r^{2})_t$,
which converts the scalar equation into a convex quadratic with a strictly
negative constant term, always yielding a unique positive root.
For the Crank--Nicolson scheme, the product-form discretization $r^{n+1}r^{n}$
preserves this quadratic structure, while conventional alternatives do not.
The resulting schemes incur the same computational cost as the original
SAV method and are proved unconditionally energy-stable.
Numerical experiments for the Cahn--Hilliard equation confirm the predicted positivity, energy stability, and convergence rates.



\end{abstract}

%===============================================================
\section{Introduction}
\label{sec1}

Gradient flows that dissipate a free energy functional $\mathcal{E}[\phi]$ are central
to modelling in physics, materials science, and image processing~\cite{BaiX23,Huang19a,Huang19b}.
Numerical schemes that inherit this dissipative structure at the discrete level
are therefore highly desirable.

We focus on free energies of the form
\begin{equation}
\mathcal{E}[\phi(\bm{x},t)]=\frac{1}{2}(\phi,\mathcal{L}\phi)
+\mathcal{E}_{_\mathcal{N}}[\phi],\quad (\bm{x},t)\in\Omega\times[0,T],
\label{eq1}
\end{equation}
where $(\cdot,\cdot)$ denotes the $L^{2}$ inner product,
$\mathcal{L}$ is a linear, self-adjoint, non-negative operator,
and $\mathcal{E}_{_\mathcal{N}}[\phi]>-C$ (i.e., bounded below) contains the
strongly nonlinear contribution.
The associated gradient flow with a mobility operator $\mathcal{G}$
(assumed to be self-adjoint, positive definite and invertible) reads
\begin{equation}
\left\{\begin{array}{l}
\phi_{t} = -\mathcal{G}\mu,\\[1mm]
\mu = \mathcal{L}\phi + \mathcal{N}(\phi),
\end{array}\right.
\quad\text{where}\quad\mathcal{N}(\phi):=\frac{\delta\mathcal{E}_{_\mathcal{N}}[\phi]}{\delta\phi},\label{eq2}
\end{equation}
subject to the initial condition $\phi(\bm{x},0)=\phi_{0}(\bm{x})$
and suitable boundary conditions (e.g., periodic or homogeneous Neumann).
Integration by parts then yields the fundamental energy dissipation law
\begin{equation}
\frac{\mathrm{d}\mathcal{E}[\phi]}{\mathrm{d}t}
=\left(\frac{\delta\mathcal{E}}{\delta\phi},\frac{\partial\phi}{\partial t}\right)
=-(\mu,\mathcal{G}\mu) \le 0.\label{eq3}
\end{equation}

Among various energy-stable methods such as convex splitting~\cite{Eyre98},
exponential time differencing~\cite{Du18,FuZ22},
invariant energy quadratization (IEQ)~\cite{XuZ19},
and Lagrange multiplier approaches~\cite{ChengQ20},
the scalar auxiliary variable (SAV) method~\cite{ShenJ18} stands out:
it yields linear, unconditionally energy-stable schemes without restrictive assumptions.
The original SAV approach introduces $r(t)=\sqrt{\mathcal{E}_{_\mathcal{N}}[\phi]+C}>0$
and rewrites the gradient flow as an equivalent augmented system.
Since its inception, the SAV framework has been refined in several directions.
The relaxed SAV~\cite{JiangM22,LiuZ24,Huang24,ShenJ26}
and weighted SAV~\cite{Huang26} variants
improve the consistency between $r^{n+1}$ and $\sqrt{\mathcal{E}_{_\mathcal{N}}[\phi^{n+1}]+C}$,
thereby enhancing accuracy for large time steps.
The exponential SAV~\cite{LiuZ20} guarantees positivity through a logarithmic reformulation,
while the generalized SAV~\cite{HuangF20,HuangF22} and gPAV~\cite{YangZ20} approaches
also preserve positivity by redefining the auxiliary variable.
A stabilized Crank--Nicolson SAV scheme~\cite{LiuZ26}
achieves positivity through a carefully designed stabilization parameter;
an upwind-SAV scheme~\cite{HuangJY23} preserves pointwise bounds
on the phase variable for degenerate Cahn--Hilliard equations
via finite-volume spatial discretization;
and extensions to PDEs with unbounded energy functionals~\cite{Kemmochi22} further enrich the framework.
A comprehensive account of the SAV framework appears in~\cite{ShenJ19};
broader surveys of phase field modelling and numerical approximations can be found in~\cite{DuF20}.

After time discretization of the original SAV system, two issues arise:
\begin{enumerate}[nosep]
  \item the computed $r^{n+1}$ is not constrained to equal
    $\sqrt{\mathcal{E}_{_\mathcal{N}}[\phi^{n+1}]+C}$, and the discrepancy
    can become significant when the time step is ``large'';
  \item $r^{n+1}>0$ cannot be guaranteed theoretically,
    contradicting the continuous definition $r(t)=\sqrt{\mathcal{E}_{_\mathcal{N}}[\phi]+C}>0$.
\end{enumerate}
The first issue has been effectively addressed by relaxed
and weighted SAV variants~\cite{JiangM22,LiuZ24,Huang24,Huang26}.
The second issue, namely guaranteeing positivity while retaining the
square-root definition $r=\sqrt{\mathcal{E}_{_\mathcal{N}}[\phi]+C}$,
has received surprisingly little attention despite the widespread use of the original SAV method.
The exponential SAV~\cite{LiuZ20} guarantees $r^{n+1}>0$
through a logarithmic change of variable,
yet the lower bound of $\ln r^{n+1}$ is not directly controlled,
which complicates convergence analysis.
The generalized SAV~\cite{HuangF20,HuangF22} also preserves positivity,
but its auxiliary variable evolution does not explicitly involve the associated nonlinear energy term.
To date, guaranteeing positivity within the original square-root definition has remained an open problem.

Preserving the original definition is of interest beyond aesthetic consistency.
The original SAV scheme remains the most widely implemented member
of the SAV family, serving as the foundation upon which relaxed, weighted
and other variants are built.
The relation $r=\sqrt{\mathcal{E}_{_\mathcal{N}}[\phi]+C}$ is the simplest
algebraic link between the auxiliary variable and the nonlinear
energy; replacing $r$ by a different function of
$\mathcal{E}_{_\mathcal{N}}[\phi]$ would alter the structure of the
scalar equation at the discrete level, generally leading to a
more complex algebraic problem than the compact quadratic form
derived in Section~\ref{sec3}.
Moreover, relaxed SAV variants take the original SAV
solution as the starting point for a post-processing correction;
if the base scheme already produces $r^{n+1}\le0$, the
subsequent correction step operates on a value that contradicts
the very definition of the auxiliary variable.
Fixing the sign-indeterminacy at the source, within the original
formulation, therefore strengthens the entire SAV framework at
no additional computational cost.

In this work, we resolve this question in two steps.
First, we diagnose the problem:
a sharp criterion for $r^{n+1}>0$ is derived,
a sufficient condition on time step size is established via an operator-norm bound,
and an explicit counterexample demonstrates that sign failure is not a pathological edge
case but a structural phenomenon that can occur for practically relevant parameters
and is particularly acute for solutions with significant high-frequency content.
Second, we resolve it at the structural level:
rather than altering the definition of
$r$, we retain the classical square-root form and reformulate the
discrete evolution from $r_t$ to $(r^{2})_t$.
This simple algebraic change converts the scalar auxiliary equation from a
linear equation with a sign-indefinite constant term into a convex quadratic
whose constant term is strictly negative,
thereby always yielding a unique positive root $r^{n+1}$ at every time step.
We construct positivity-preserving SAV schemes for both the backward Euler
and Crank--Nicolson discretizations,
prove their unconditional energy stability,
and show that they incur exactly the same computational cost as the
original SAV method.
For the Crank--Nicolson scheme, the product-form discretization $r^{n+1}r^{n}$
at the half-time level is the key design choice:
it preserves the quadratic structure essential for the positivity analysis,
while natural alternatives such as the arithmetic mean
$\frac{1}{2}[(r^{n+1})^{2}+(r^{n})^{2}]$ or the squared mean
$(\frac{r^{n+1}+r^{n}}{2})^{2}$ lead to cubic equations
that admit no comparable guarantee.
Numerical experiments confirm the predicted positivity, energy stability, and convergence rates.

The remainder of the paper is organized as follows.
Section~\ref{sec2} reviews the original SAV method, derives a sufficient
condition for positivity, constructs an explicit counterexample,
and diagnoses the structural origin of the sign-indeterminacy.
Section~\ref{sec3} develops the positivity-preserving backward Euler
and Crank--Nicolson SAV schemes and proves their unconditional energy stability.
Section~\ref{sec4} presents numerical results, and Section~\ref{sec5} gives concluding remarks.

%===============================================================
\section{The original SAV method and its sign-indeterminacy}
\label{sec2}

Following the original SAV approach~\cite{ShenJ18}, the free energy~\eqref{eq1}
can be rewritten as
\begin{equation}
\mathcal{E}[\phi,r] = \frac{1}{2}(\phi,\mathcal{L}\phi) + r^{2} - C,\quad\text{where}\quad
r(t) = \sqrt{\mathcal{E}_{_\mathcal{N}}[\phi] + C} > 0,
\label{eq5}
\end{equation}
and the gradient flow system~\eqref{eq2} is rewritten as
\begin{equation}
\left\{\begin{array}{l}
\phi_{t} = -\mathcal{G}\mu,\\[2mm]
\mu = \mathcal{L}\phi + \dfrac{r}{\sqrt{\mathcal{E}_{_\mathcal{N}}[\phi]+C}}\,\mathcal{N}(\phi),\\[5mm]
r_{t} = \dfrac{1}{2\sqrt{\mathcal{E}_{_\mathcal{N}}[\phi]+C}}\bigl(\mathcal{N}(\phi),\,\phi_{t}\bigr).
\end{array}\right.
\label{eq6}
\end{equation}

Discretizing~\eqref{eq6} with the backward Euler method yields
\begin{equation}
\left\{\begin{array}{l}
\dfrac{\phi^{n+1}-\phi^{n}}{\tau}=-\mathcal{G}\mu^{n+1},\\[3mm]
\mu^{n+1}=\mathcal{L}\phi^{n+1}
  +\dfrac{r^{n+1}}{\sqrt{\mathcal{E}_{_\mathcal{N}}[\phi^{n}]+C}}\mathcal{N}(\phi^{n}),\\[5mm]
\dfrac{r^{n+1}-r^{n}}{\tau}
 =\dfrac{1}{2\sqrt{\mathcal{E}_{_\mathcal{N}}[\phi^{n}]+C}}\Big(\mathcal{N}(\phi^{n}),
   \dfrac{\phi^{n+1}-\phi^{n}}{\tau}\Big),
\end{array}\right.\label{eq7}
\end{equation}
where $\tau>0$ is the time step size and $\chi^n$
(for $\chi=\phi,\mu,r$) represents the numerical approximation
of $\chi(t)$ at $t=t^n:=n\tau$.
Eliminating $\mu^{n+1}$ gives
\begin{equation}
(\mathcal{I} + \tau\mathcal{G}\mathcal{L})\phi^{n+1}
= \phi^{n} - \frac{\tau\,\mathcal{G}\mathcal{N}(\phi^{n})}{\sqrt{\mathcal{E}_{_\mathcal{N}}[\phi^{n}] + C}}\,r^{n+1}.
\label{eq8}
\end{equation}
By linear superposition, we set
\begin{equation}
\phi^{n+1} = p^{n+1} + r^{n+1} q^{n+1},
\label{eq9}
\end{equation}
where $p^{n+1}$ and $q^{n+1}$ solve the decoupled linear equations
\begin{equation}
(\mathcal{I} + \tau\mathcal{G}\mathcal{L})\,p^{n+1} = \phi^{n},\qquad
(\mathcal{I} + \tau\mathcal{G}\mathcal{L})\,q^{n+1}
= -\frac{\tau\,\mathcal{G}\mathcal{N}(\phi^{n})}{\sqrt{\mathcal{E}_{_\mathcal{N}}[\phi^{n}] + C}}.
\label{eq10}
\end{equation}
Substituting~\eqref{eq9} into the third equation of~\eqref{eq7}
leads to a scalar equation for $r^{n+1}$:
\begin{equation}
h(r^{n+1}) :=
\underbrace{\left[1 - \frac{(\mathcal{N}(\phi^{n}),\, q^{n+1})}{2\sqrt{\mathcal{E}_{_\mathcal{N}}[\phi^{n}]+C}}\right]}_{=:\mathcal{A}^{n+1}} r^{n+1}
-\underbrace{\left[r^{n} + \frac{(\mathcal{N}(\phi^{n}),\, p^{n+1}-\phi^{n})}{2\sqrt{\mathcal{E}_{_\mathcal{N}}[\phi^{n}]+C}}\right]}_{=:\mathcal{B}^{n+1}} = 0.
\label{eq11}
\end{equation}
From~\eqref{eq10} and the invertibility of $\mathcal{G}$, it follows that
\begin{equation}
- \frac{(\mathcal{N}(\phi^{n}),\, q^{n+1})}{\sqrt{\mathcal{E}_{_\mathcal{N}}[\phi^{n}] + C}}
= \frac{1}{\tau}\bigl((\mathcal{G}^{-1} + \tau\mathcal{L})q^{n+1},\, q^{n+1}\bigr) \ge 0,
\label{eq12}
\end{equation}
which implies that $\mathcal{A}^{n+1}\ge 1$,
so $h(\cdot)$ is a strictly increasing linear function and $r^{n+1}=\mathcal{B}^{n+1}/\mathcal{A}^{n+1}$.
Consequently $\operatorname{sgn}(r^{n+1})=\operatorname{sgn}(\mathcal{B}^{n+1})$.
However, the constant term $\mathcal{B}^{n+1}$ is the sum of the positive
quantity $r^{n}$ and a perturbation involving $\mathcal{N}(\phi^{n})$ whose
sign is indefinite; when the perturbation is sufficiently negative it can
overwhelm $r^{n}$, so $r^{n+1}>0$ cannot be guaranteed.

The Crank--Nicolson discretization of~\eqref{eq6} yields a scalar
equation of the same linear form as~\eqref{eq11}, with
$\mathcal{N}(\phi^{n})$ and $\mathcal{E}_{_\mathcal{N}}[\phi^{n}]$
replaced by their counterparts evaluated at an explicit second-order
approximation $\phi^{n+\frac{1}{2}}_{\ast}$ of $\phi(t^{n+\frac{1}{2}})$.
Its constant term is likewise sign-indefinite.
Thus the original SAV schemes, both backward Euler and Crank--Nicolson, offer
no theoretical guarantee that $r^{n+1}>0$, an inconsistency with the continuous
definition $r(t)>0$ that may undermine long-time numerical robustness.

\subsection{A sufficient condition for positivity}

To quantify precisely when positivity can fail, we now derive an explicit
sufficient condition on the time step $\tau$ that guarantees $r^{n+1}>0$.
Since $\operatorname{sgn}(r^{n+1})=\operatorname{sgn}(\mathcal{B}^{n+1})$,
it suffices to determine when $\mathcal{B}^{n+1}>0$.

From~\eqref{eq10} we have $p^{n+1}=(\mathcal{I}+\tau\mathcal{G}\mathcal{L})^{-1}\phi^{n}$, hence
\begin{equation}
p^{n+1}-\phi^{n}=[(\mathcal{I}+\tau\mathcal{G}\mathcal{L})^{-1}-\mathcal{I}]\phi^{n}
=-\tau(\mathcal{I}+\tau\mathcal{G}\mathcal{L})^{-1}\mathcal{G}\mathcal{L}\,\phi^{n}.
\label{eqrp2}
\end{equation}
Substituting this identity into the definition of $\mathcal{B}^{n+1}$ yields the \textbf{sharp criterion}
\begin{equation}
r^{n+1}>0 \;\Longleftrightarrow\; 2r^{n}\sqrt{\mathcal{E}_{_\mathcal{N}}[\phi^{n}]+C}
\;>\,\tau\,\Bigl(\mathcal{N}(\phi^{n}),\;
(\mathcal{I}+\tau\mathcal{G}\mathcal{L})^{-1}\mathcal{G}\mathcal{L}\,\phi^{n}\Bigr).\label{eq13}
\end{equation}
This is an exact reformulation of $\operatorname{sgn}(\mathcal{B}^{n+1})>0$;
the difficulty is that the inner product on the right-hand side has no a priori
sign and depends nonlinearly on $\tau$ through the resolvent
$(\mathcal{I}+\tau\mathcal{G}\mathcal{L})^{-1}$.
This criterion depends only on $\phi^{n}$,
$r^{n}$, the operators $\mathcal{G},\mathcal{L}$, and the time step
$\tau$; hence \eqref{eq13} provides a computable necessary and sufficient
condition for $r^{n+1}>0$ using only data available at step~$n$.

To obtain a practically verifiable sufficient condition, we bound the inner
product from above.  Define the symmetrized operator
\begin{equation}
\mathcal{S}:=\mathcal{G}^{\frac{1}{2}}\mathcal{L}\mathcal{G}^{\frac{1}{2}},
\label{eq14}
\end{equation}
which is self-adjoint and positive semidefinite because $\mathcal{L}$ and
$\mathcal{G}^{\frac{1}{2}}$ are.
(Here $\mathcal{G}^{\frac{1}{2}}$ denotes the unique self-adjoint positive
definite square root of $\mathcal{G}$, and
$\mathcal{G}^{-\frac{1}{2}}:=(\mathcal{G}^{\frac{1}{2}})^{-1}$.)
A direct computation gives the similarity transform
$\mathcal{G}\mathcal{L}= \mathcal{G}^{\frac{1}{2}}\mathcal{S}\,\mathcal{G}^{-\frac{1}{2}}$,
and therefore
\begin{align}
(\mathcal{I}+\tau\mathcal{G}\mathcal{L})^{-1}\mathcal{G}\mathcal{L}
&=\bigl(\mathcal{G}^{\frac{1}{2}}\mathcal{G}^{-\frac{1}{2}}
        +\tau\mathcal{G}^{\frac{1}{2}}\mathcal{S}\mathcal{G}^{-\frac{1}{2}}\bigr)^{-1}
   \mathcal{G}^{\frac{1}{2}}\mathcal{S}\mathcal{G}^{-\frac{1}{2}}\nonumber\\
&=\bigl(\mathcal{G}^{\frac{1}{2}}(\mathcal{I}+\tau \mathcal{S})\mathcal{G}^{-\frac{1}{2}}\bigr)^{-1}
   \mathcal{G}^{\frac{1}{2}}\mathcal{S}\mathcal{G}^{-\frac{1}{2}}\nonumber\\
&=\mathcal{G}^{\frac{1}{2}}(\mathcal{I}+\tau \mathcal{S})^{-1}\mathcal{S}\,\mathcal{G}^{-\frac{1}{2}}.
\label{eq15}
\end{align}
The operator $(\mathcal{I}+\tau\mathcal{S})^{-1}\mathcal{S}$ is self-adjoint
(because $\mathcal{S}$ is), and the spectral theorem together with the
monotonicity of $\lambda\mapsto\lambda/(1+\tau\lambda)$ on $[0,\infty)$ gives
\begin{equation}
\|(\mathcal{I}+\tau \mathcal{S})^{-1}\mathcal{S}\|
= \sup_{\lambda\in\sigma(\mathcal{S})}\frac{\lambda}{1+\tau\lambda}
= \frac{\rho}{1+\tau\rho},
\qquad \rho:=\|\mathcal{S}\|=\sup_{u\neq 0}\frac{\|\mathcal{S}u\|}{\|u\|}.
\label{eq16}
\end{equation}

Now we estimate the critical inner product:
\begin{align}
\bigl(\mathcal{N}(\phi^{n}),\,(\mathcal{I}+\tau\mathcal{G}\mathcal{L})^{-1}\mathcal{G}\mathcal{L}\,\phi^{n}\bigr)
&=\bigl(\mathcal{G}^{\frac{1}{2}}\mathcal{N}(\phi^{n}),\,
        (\mathcal{I}+\tau \mathcal{S})^{-1}\mathcal{S}\;
        \mathcal{G}^{-\frac{1}{2}}\phi^{n}\bigr)\nonumber\\
&\le \|(\mathcal{I}+\tau \mathcal{S})^{-1}\mathcal{S}\|\;
     \|\mathcal{G}^{\frac{1}{2}}\mathcal{N}(\phi^{n})\|\,
     \|\mathcal{G}^{-\frac{1}{2}}\phi^{n}\|\nonumber\\
&=\frac{\rho}{1+\tau\rho}\,
   \|\mathcal{G}^{\frac{1}{2}}\mathcal{N}(\phi^{n})\|\,
   \|\mathcal{G}^{-\frac{1}{2}}\phi^{n}\|.
\label{eq17}
\end{align}
The equality uses \eqref{eq15} and the self-adjointness of
$\mathcal{G}^{\frac{1}{2}}$, the inequality is Cauchy--Schwarz
combined with the operator norm, and the final step applies the
spectral bound \eqref{eq16}.

Combining~\eqref{eq13} with the estimate~\eqref{eq17}, we obtain the following
sufficient condition.

\begin{prop}[Sufficient condition for positivity]\label{prop1}
Consider the original SAV-BE scheme \eqref{eq7},
let $\mathcal{S}:=\mathcal{G}^{\frac{1}{2}}\mathcal{L}\mathcal{G}^{\frac{1}{2}}$
and $\rho:=\|\mathcal{S}\|$.
Assume $r^{n}>0$ and define
\begin{equation}
L:=2r^{n}\sqrt{\mathcal{E}_{_\mathcal{N}}[\phi^{n}]+C},\qquad
M:=\|\mathcal{G}^{\frac{1}{2}}\mathcal{N}(\phi^{n})\|\,
    \|\mathcal{G}^{-\frac{1}{2}}\phi^{n}\|.
\label{eq18}
\end{equation}
\begin{enumerate}[label={(\roman*)},nosep]
  \item If $L\ge M$, then $r^{n+1}>0$ for all $\tau>0$.
  \item If $L<M$, then $r^{n+1}>0$ for all $\tau$ satisfying
  \begin{equation}
  \tau<\frac{1}{\rho}\frac{L}{M-L}
       =\frac{1}{\rho}\,
         \frac{2r^{n}\sqrt{\mathcal{E}_{_\mathcal{N}}[\phi^{n}]+C}}
              {\|\mathcal{G}^{\frac{1}{2}}\mathcal{N}(\phi^{n})\|\,
               \|\mathcal{G}^{-\frac{1}{2}}\phi^{n}\|
               -2r^{n}\sqrt{\mathcal{E}_{_\mathcal{N}}[\phi^{n}]+C}}.
  \label{eq19}
  \end{equation}
\end{enumerate}
\end{prop}

\begin{proof}
From the definition of $\mathcal{B}^{n+1}$ and~\eqref{eqrp2},
\begin{align}
\mathcal{B}^{n+1}
&=r^{n}+\frac{(\mathcal{N}(\phi^{n}),\,p^{n+1}-\phi^{n})}
                {2\sqrt{\mathcal{E}_{_\mathcal{N}}[\phi^{n}]+C}}\nonumber\\
&=\frac{1}{2\sqrt{\mathcal{E}_{_\mathcal{N}}[\phi^{n}]+C}}
   \Bigl[2r^{n}\sqrt{\mathcal{E}_{_\mathcal{N}}[\phi^{n}]+C}
        -\tau\bigl(\mathcal{N}(\phi^{n}),\,
          (\mathcal{I}+\tau\mathcal{G}\mathcal{L})^{-1}\mathcal{G}\mathcal{L}\,\phi^{n}\bigr)\Bigr]\nonumber\\
&\ge\frac{1}{2\sqrt{\mathcal{E}_{_\mathcal{N}}[\phi^{n}]+C}}
   \Bigl[L - \frac{\tau\rho}{1+\tau\rho}\,M\Bigr],
\label{eq20}
\end{align}
where the inequality follows from~\eqref{eq17}.
Since $\mathcal{A}^{n+1}\ge 1>0$,
$\operatorname{sgn}(r^{n+1})=\operatorname{sgn}(\mathcal{B}^{n+1})$.
A sufficient condition for $\mathcal{B}^{n+1}>0$ is therefore
$L > \frac{\tau\rho}{1+\tau\rho}M$, which is equivalent to
\begin{equation}
\tau\rho\,(M-L) < L.
\label{eq21}
\end{equation}
If $M\le L$, the left-hand side of~\eqref{eq21} is non-positive while $L>0$,
so the inequality holds for every $\tau>0$, giving case~(i).
If $M>L$, then~\eqref{eq21} is equivalent to
$\tau < \frac{1}{\rho}\frac{L}{M-L}$, which is precisely~\eqref{eq19}.
\end{proof}

In the ODE reduction ($\mathcal{G}=\mathcal{I}$,
$\mathcal{L}=\alpha\mathcal{I}$, $|\Omega|=1$),
the operator $(\mathcal{I}+\tau\mathcal{G}\mathcal{L})^{-1}\mathcal{G}\mathcal{L}$
reduces to $\frac{\alpha}{1+\tau\alpha}$, and
\eqref{eq13} becomes
$2r^{n}\sqrt{\mathcal{E}_{_\mathcal{N}}[\phi^{n}]+C}
> \frac{\tau\alpha}{1+\tau\alpha}\,\mathcal{N}(\phi^{n})\phi^{n}$
The sign of the right-hand side is governed entirely by
$\mathcal{N}(\phi^{n})\phi^{n}$.
If $\phi^{n}$ is bounded, this term can be estimated directly.
When $\mathcal{N}(\phi^{n})\phi^{n}\le 0$, the right-hand side is
non-positive and $r^{n+1}>0$ holds unconditionally.
When $\mathcal{N}(\phi^{n})\phi^{n}>0$, the right-hand side
is positive and, for sufficiently large $\tau\alpha$, may reach or
exceed the left-hand side, leading to $r^{n+1}\le 0$; this is the
mechanism underlying Example~\ref{exa1} below.
For PDE models the operator
$(\mathcal{I}+\tau\mathcal{G}\mathcal{L})^{-1}\mathcal{G}\mathcal{L}$
is non-local in space, and a simple pointwise estimate as in the
ODE case is no longer available; the operator-norm bound in
Proposition~\ref{prop1} provides one way to control this non-local
coupling, at the cost of being conservative.

\begin{rem}
Proposition~\ref{prop1} provides a \emph{sufficient} but not necessary
condition.  Two successive bounds intervene: the Cauchy--Schwarz inequality
in~\eqref{eq17} and the replacement of the spectral function
$\lambda\mapsto\lambda/(1+\tau\lambda)$ by its supremum $\rho/(1+\tau\rho)$.
Consequently, the time-step restriction~\eqref{eq19} may be considerably more
pessimistic than the sharp threshold~\eqref{eq13}.  Its principal value is
theoretical: it confirms that for any fixed $\phi^{n}$ satisfying the
hypotheses there always exists a sufficiently small $\tau$ for which the
original SAV scheme preserves positivity, and it identifies the product
$\|\mathcal{G}^{\frac{1}{2}}\mathcal{N}(\phi^{n})\|\,
\|\mathcal{G}^{-\frac{1}{2}}\phi^{n}\|$ as the key quantity controlling the
sign of $\mathcal{B}^{n+1}$.  When $\mathcal{G}$ and $\mathcal{L}$ commute
(as in Fourier spectral discretizations of constant-coefficient operators),
$\mathcal{G}^{\frac{1}{2}}\mathcal{L}\mathcal{G}^{\frac{1}{2}}
 =\mathcal{G}\mathcal{L}$, the weighted norms reduce to standard $L^{2}$-norms,
and $\rho=\|\mathcal{G}\mathcal{L}\|$.
\end{rem}

\subsection{An explicit counterexample}

To exhibit the failure concretely, and for ease of illustration,
we consider the ODE reduction
$\mathcal{G}=\mathcal{I}$, $\mathcal{L}=\alpha\mathcal{I}$ ($\alpha>0$),
where the state $\phi$ is a scalar and $|\Omega|=1$, so that
$\mathcal{E}_{_\mathcal{N}}[\phi]=\int_{\Omega}F(\phi)\,\mathrm{d}\bm{x}=F(\phi)$
with the standard double-well potential
\begin{equation}
F(\phi)=\frac{1}{4}(\phi^{2}-1)^{2},\qquad
\mathcal{N}(\phi)=F'(\phi)=\phi(\phi^{2}-1).
\label{eq22}
\end{equation}
We consider the first time step ($n=0$), where the initial data are
consistent: $r^{0}=\sqrt{\mathcal{E}_{_\mathcal{N}}[\phi^{0}]+C}$.

Under these simplifications we have $p^{1}=\frac{\phi^{0}}{1+\tau\alpha}$,
$\mathcal{A}^{1}=1+\frac{\tau}{2(1+\tau\alpha)}\frac{(\mathcal{N}(\phi^{0}))^{2}}{\mathcal{E}_{_\mathcal{N}}[\phi^{0}]+C}$,
and the sharp criterion~\eqref{eq13} reduces to
\begin{equation}
r^{1}>0\;\Longleftrightarrow\; 2\bigl(\mathcal{E}_{_\mathcal{N}}[\phi^{0}]+C\bigr)
\;>\,\frac{\tau\alpha}{1+\tau\alpha}\,\mathcal{N}(\phi^{0})\,\phi^{0}.
\label{eq23}
\end{equation}

\begin{exa}[Positivity failure at moderate time step]\label{exa1}
Choose
\begin{equation}
\alpha = 12,\qquad \tau = 0.1,\qquad C = 0,\qquad \phi^{0} = 2.
\label{eq25}
\end{equation}
Then
\begin{equation}
\mathcal{E}_{_\mathcal{N}}[\phi^{0}]=\frac{9}{4},\qquad
r^{0}=\frac{3}{2}> 0,\qquad
\mathcal{N}(\phi^{0})=6,\qquad
\mathcal{N}(\phi^{0})\,\phi^{0}=12.
\label{eq27}
\end{equation}
The coefficient and constant term work out to clean fractions:
\begin{equation}
\mathcal{A}^{1}=\frac{15}{11},\qquad
\mathcal{B}^{1}=-\frac{15}{22}.\label{eq29}
\end{equation}
Hence
\begin{equation}
r^{1} = \frac{\mathcal{B}^{1}}{\mathcal{A}^{1}}= -\frac{1}{2} < 0.\label{eq30}
\end{equation}
A single backward Euler step with $\tau=0.1$ therefore destroys positivity even though $r^{0}=1.5>0$.
\end{exa}

\begin{rem}[Threshold analysis]\label{rem:threshold}
More generally, for the ODE reduction with $C=0$, note that
$\mathcal{N}(\phi^{0})\phi^{0}\le 0$ implies $r^{1}>0$ unconditionally
by~\eqref{eq23}.  When $\mathcal{N}(\phi^{0})\phi^{0}>0$, the non-positivity
condition $r^{1}\le 0$ becomes
\begin{equation}
\tau \ge \frac{1}{\alpha}\,\frac{(\phi^{0})^{2}-1}{(\phi^{0})^{2}+1}.\label{eq31}
\end{equation}
Whenever $|\phi^{0}|>1$, the right-hand side is strictly positive, so a
finite $\tau$ (for given $\alpha$) or, equivalently, a finite $\alpha$ (for
given $\tau$) always triggers the sign failure.  For $\phi^{0}=2$,
$\frac{(\phi^{0})^{2}-1}{(\phi^{0})^{2}+1} = 0.6$, so the
condition becomes $\tau\alpha \ge 0.6$.  Thus
\begin{equation}
\tau = 10^{-3} \;\Rightarrow\; \alpha \ge 600;\qquad
\tau = 10^{-5} \;\Rightarrow\; \alpha \ge 6\times 10^{4}.
\label{eq32}
\end{equation}
The ODE threshold \eqref{eq31} highlights the role of the
product $\tau\alpha$: sign failure requires both a sufficiently
large stiffness $\alpha$ and a sufficiently large $\tau$.
In a PDE setting the operator
$(\mathcal{I}+\tau\mathcal{G}\mathcal{L})^{-1}\mathcal{G}\mathcal{L}$
couples spatial scales in a more intricate way, and a direct
translation of the ODE threshold via $\|\mathcal{G}\mathcal{L}\|$
is overly pessimistic because it ignores the decay of Fourier
coefficients at high wave numbers for smooth solutions.
We emphasise that the above ODE counterexample is presented for clarity of exposition;
the same mechanism extends to PDEs, where high-frequency
modes play the role of large stiffness.
A PDE counterpart of this counterexample is presented in
Section~\ref{sec4}.
\end{rem}

\subsection{Structural nature of the sign-indeterminacy}

The failure analysed above is not a numerical artefact of large time steps or
poorly chosen parameters; it is a \emph{structural} consequence of
discretizing $r_{t}$ directly.

In the original SAV formulation~\eqref{eq6}, discretizing
$r_{t}$ directly by backward Euler (or Crank--Nicolson) yields a linear scalar equation
$\mathcal{A}^{n+1}r^{n+1}-\mathcal{B}^{n+1}=0$ whose constant term
$\mathcal{B}^{n+1}$ is the sum of the positive quantity $r^{n}$ and a
perturbation driven by $\mathcal{N}(\phi^{n})$.
As the sharp criterion \eqref{eq13} shows, the sign of $r^{n+1}$ is
controlled by the competition between $r^{n}$ and this perturbation,
which depends nonlinearly on the spatial structure of $\phi^{n}$ through
the operator $(\mathcal{I}+\tau\mathcal{G}\mathcal{L})^{-1}\mathcal{G}\mathcal{L}$.
Nothing in the algebraic structure of the linear equation prevents this
perturbation from overcoming $r^{n}$.

The root cause is evident from the structure of
$h(r^{n+1})=\mathcal{A}^{n+1}r^{n+1}-\mathcal{B}^{n+1}=0$:
while the slope $\mathcal{A}^{n+1}\ge 1$ is strictly positive, the intercept
$-\mathcal{B}^{n+1}$ has no fixed sign.
Geometrically, this means the line $h(\cdot)$ can cross the horizontal axis
at any point, positive, zero, or negative, depending on the data $\phi^{n}$.

The resolution, developed in Section~\ref{sec3}, is to reformulate the
discrete evolution from $r_{t}$ to $(r^{2})_{t}$.
This converts the algebraic structure completely:
from a \emph{linear} equation whose constant term has an \emph{indefinite} sign,
to a \emph{convex quadratic} whose constant term $-(r^{n})^{2}$ is
\emph{strictly negative}.
Geometrically, a convex quadratic with $h(0)<0$ and $h(+\infty)=+\infty$ must
cross the positive $r$-axis exactly once, irrespective of the data.
This structural change is the key to the unconditional positivity guarantees
proved in Theorems~\ref{thmf1} and~\ref{thmf3}.

We close this section with a remark on why the sign-indeterminacy
has largely escaped attention despite the widespread use of the original SAV method.
In the ODE reduction, the sign failure requires $|\phi^{n}|>1$ for
the double-well potential, a condition that may not arise when the
solution stays within the metastable wells.
In PDE simulations, the sign of the perturbation in \eqref{eq13}
depends on the interplay between the amplitude of $\phi^{n}$ and
the spectral distribution of the operator
$(\mathcal{I}+\tau\mathcal{G}\mathcal{L})^{-1}\mathcal{G}\mathcal{L}$:
even a perfectly smooth function can trigger the failure if it
carries sufficient amplitude at wave numbers for which the
eigenvalues of $\mathcal{G}\mathcal{L}$ are large relative to $\tau^{-1}$.

%===============================================================
\section{Positivity-preserving SAV schemes}
\label{sec3}

As analysed in the previous section, the sign-indeterminacy of the original SAV
scheme originates from discretizing $r_{t}$ directly, which yields a linear
scalar equation whose constant term is not sign-definite.
The key observation is that $r_{t} = (r^{2})_{t}/(2r)$ for $r>0$;
reformulating the continuous equations in terms of $(r^{2})_{t}$
and discretizing afterwards converts the scalar equation into a convex quadratic
with a strictly negative constant term $-(r^{n})^{2}$.

In contrast to the sharp criterion \eqref{eq13}, where positivity
must be verified at each step and may fail, the quadratic reformulation
guarantees $r^{n+1}>0$ unconditionally by its algebraic structure:
a strictly negative constant term together with a strictly convex
leading coefficient forces a unique positive root regardless of the data.
Thus the sign-indeterminacy, which Section~\ref{sec2} traced to the
linear structure of the original discretization, is resolved here not
by tuning numerical parameters but by a structural change in the
algebraic form of the scalar equation: the discrete evolution is
rewritten from $r_{t}$ to $(r^{2})_{t}$, a change that is
mathematically trivial at the continuous level yet decisive at the
discrete level.
We now develop this idea for the backward Euler (BE)
and Crank--Nicolson (CN) discretizations.

\subsection{Backward Euler scheme}

Since $r>0$, the chain rule gives $(r^2)_t=2 r\cdot r_t$.
Replacing $r_t$ in~\eqref{eq6} accordingly yields the equivalent formulation
\begin{equation}
\left\{\begin{array}{l}
\phi_{t} = -\mathcal{G}\mu,\\[2mm]
\mu = \mathcal{L}\phi + \dfrac{r}{\sqrt{\mathcal{E}_{_\mathcal{N}}[\phi]+C}}\,\mathcal{N}(\phi),\\[5mm]
(r^{2})_{t}
= \dfrac{r}{\sqrt{\mathcal{E}_{_\mathcal{N}}[\phi]+C}}\bigl(\mathcal{N}(\phi),\,\phi_{t}\bigr).
\end{array}\right.
\label{eqb1}
\end{equation}
Discretizing~\eqref{eqb1} with the backward Euler method gives the SAV-BE scheme:
\begin{equation}
\left\{\begin{array}{l}
\dfrac{\phi^{n+1}-\phi^{n}}{\tau} = -\mathcal{G}\mu^{n+1},\\[3mm]
\mu^{n+1} = \mathcal{L}\phi^{n+1}
  + \dfrac{r^{n+1}}{\sqrt{\mathcal{E}_{_\mathcal{N}}[\phi^{n}]+C}}\,\mathcal{N}(\phi^{n}),\\[5mm]
\dfrac{(r^{n+1})^{2}-(r^{n})^{2}}{\tau}
= \dfrac{r^{n+1}}{\sqrt{\mathcal{E}_{_\mathcal{N}}[\phi^{n}]+C}}
\Bigl(\mathcal{N}(\phi^{n}),\,\dfrac{\phi^{n+1}-\phi^{n}}{\tau}\Bigr),
\end{array}\right.
\label{eqb2}
\end{equation}
where $\tau>0$ is the time step size and $\phi^n$, $\mu^n$, $r^n$ denote the
numerical approximations at $t=t^n:=n\tau$.

Eliminating $\mu^{n+1}$ yields
\begin{equation}
(\mathcal{I}+\tau\mathcal{G}\mathcal{L})\phi^{n+1}
= \phi^{n} - \frac{\tau\,\mathcal{G}\mathcal{N}(\phi^{n})}{\sqrt{\mathcal{E}_{_\mathcal{N}}[\phi^{n}]+C}}\,r^{n+1}.
\label{eqb3}
\end{equation}
Setting $\phi^{n+1}=p^{n+1}+r^{n+1}q^{n+1}$, we define $p^{n+1}$ and $q^{n+1}$
as the solutions of
\begin{equation}
(\mathcal{I}+\tau\mathcal{G}\mathcal{L})\,p^{n+1} = \phi^{n},\qquad
(\mathcal{I}+\tau\mathcal{G}\mathcal{L})\,q^{n+1}
= -\frac{\tau\,\mathcal{G}\mathcal{N}(\phi^{n})}{\sqrt{\mathcal{E}_{_\mathcal{N}}[\phi^{n}]+C}}.
\label{eqb4}
\end{equation}
Substituting $\phi^{n+1}$ into the third equation of~\eqref{eqb2} gives
\begin{align}
h_{_{\rm BE}}(r^{n+1}) :=&
\left[1 - \frac{(\mathcal{N}(\phi^{n}),\,q^{n+1})}{\sqrt{\mathcal{E}_{_\mathcal{N}}[\phi^{n}]+C}}\right](r^{n+1})^{2}
- \frac{(\mathcal{N}(\phi^{n}),\,p^{n+1}-\phi^{n})}{\sqrt{\mathcal{E}_{_\mathcal{N}}[\phi^{n}]+C}}\,r^{n+1}\nonumber\\[1mm]
&- (r^{n})^{2}=0.
\label{eqb5}
\end{align}

In contrast to the linear equation~\eqref{eq11} in the original SAV method,
\eqref{eqb5} is a convex quadratic with leading coefficient at least~$1$
and constant term $-(r^{n})^{2}<0$ whenever $r^{n}>0$.
Together, these properties guarantee a unique positive root.

\begin{thm}\label{thmf1}
The SAV-BE scheme~\eqref{eqb2} admits a unique solution with $r^{n+1}>0$.
\end{thm}

\begin{proof}
From~\eqref{eqb4} and the same derivation as in~\eqref{eq12}, the quadratic coefficient
in~\eqref{eqb5} satisfies
\begin{equation}
1-\frac{(\mathcal{N}(\phi^{n}),\,q^{n+1})}{\sqrt{\mathcal{E}_{_\mathcal{N}}[\phi^{n}]+C}}
=1+\frac{1}{\tau}\bigl((\mathcal{G}^{-1}+\tau\mathcal{L})q^{n+1},\,q^{n+1}\bigr)\ge 1,\label{thmf1a}
\end{equation}
hence $h_{_{\rm BE}}$ is strictly convex and $\lim\limits_{r\to+\infty}h_{_{\rm BE}}(r)=+\infty$.

We prove $r^{n}>0$ for all $n$ by induction.
The base case $r^{0}=\sqrt{\mathcal{E}_{_\mathcal{N}}[\phi^{0}]+C}>0$ holds by definition.
Assume $r^{n}>0$ for some $n\ge 0$. Then $h_{_{\rm BE}}(0)=-(r^{n})^{2}<0$.
Strict convexity together with $h_{_{\rm BE}}(0)<0$ and $\lim\limits_{r\to+\infty}h_{_{\rm BE}}(r)=+\infty$
implies that $h_{_{\rm BE}}(\cdot)$ has exactly one positive root.
Selecting this root gives $r^{n+1}>0$, completing the induction.
\end{proof}

Since~\eqref{eqb5} has a negative constant term, its unique positive root
is obtained directly from the quadratic formula,
which provides an explicit guarantee of positivity.

\begin{thm}\label{thmf2}
The SAV-BE scheme~\eqref{eqb2} unconditionally satisfies the energy dissipation law
\begin{equation}
\frac{\mathcal{E}[\phi^{n+1},r^{n+1}]-\mathcal{E}[\phi^{n},r^{n}]}{\tau}
\le -(\mu^{n+1},\mathcal{G}\mu^{n+1}) \le 0,
\label{eqth1}
\end{equation}
where
\begin{equation}
\mathcal{E}[\phi^{n},r^{n}] = \frac{1}{2}(\phi^{n},\mathcal{L}\phi^{n}) + (r^{n})^{2} -C.
\label{eqth1b}
\end{equation}
\end{thm}

\begin{proof}
Taking the $L^{2}$-inner product of~\eqref{eqb2}$_{1}$ with $\mu^{n+1}$
and of~\eqref{eqb2}$_{2}$ with $(\phi^{n+1}-\phi^{n})$, we obtain
\begin{align}
&-\tau(\mu^{n+1},\mathcal{G}\mu^{n+1})\nonumber\\
=&\,(\mathcal{L}\phi^{n+1}, \phi^{n+1}-\phi^{n})
  +\frac{r^{n+1}}{\sqrt{\mathcal{E}_{_\mathcal{N}}[\phi^{n}]+C}}(\mathcal{N}(\phi^{n}), \phi^{n+1}-\phi^{n})\nonumber\\
=&\,\frac{1}{2}\bigl[(\mathcal{L}\phi^{n+1},\phi^{n+1}) - (\mathcal{L}\phi^{n},\phi^{n})
   +(\mathcal{L}(\phi^{n+1}-\phi^{n}), \phi^{n+1}-\phi^{n})\bigr]
   + (r^{n+1})^{2} - (r^{n})^{2}\nonumber\\
\ge&\,\frac{1}{2}(\mathcal{L}\phi^{n+1},\phi^{n+1})+[(r^{n+1})^{2}-C]
   -\frac{1}{2}(\mathcal{L}\phi^{n},\phi^{n})-[(r^{n})^{2}-C],\label{eqth2}
\end{align}
where the second equality follows from multiplying~\eqref{eqb2}$_{3}$ by $\tau$.
The right-hand side of the inequality~\eqref{eqth2} is precisely
$\mathcal{E}[\phi^{n+1},r^{n+1}]-\mathcal{E}[\phi^{n},r^{n}]$.
Dividing by $\tau$ and rearranging gives the energy dissipation law~\eqref{eqth1}.
\end{proof}

\subsection{Crank--Nicolson scheme}

For second-order temporal accuracy, we rewrite the coefficient
$r/\sqrt{\mathcal{E}_{_\mathcal{N}}[\phi]+C}$ as $r^{2}/(\mathcal{E}_{_\mathcal{N}}[\phi]+C)$,
a form better suited for discretization at the half-time level
$t^{n+\frac{1}{2}}:=(n+\frac{1}{2})\tau$.  The equivalent continuous formulation then reads
\begin{equation}
\left\{\begin{array}{l}
\phi_{t} = -\mathcal{G}\mu,\\[2mm]
\mu = \mathcal{L}\phi + \dfrac{r^{2}}{\mathcal{E}_{_\mathcal{N}}[\phi]+C}\,\mathcal{N}(\phi),\\[4mm]
(r^{2})_{t}= \dfrac{r^{2}}{\mathcal{E}_{_\mathcal{N}}[\phi]+C}\bigl(\mathcal{N}(\phi),\,\phi_{t}\bigr).
\end{array}\right.\label{eqc1}
\end{equation}
Applying the Crank--Nicolson method yields
\begin{equation}
\left\{\begin{array}{l}
\dfrac{\phi^{n+1}-\phi^{n}}{\tau} = -\mathcal{G}\mu^{n+\frac{1}{2}},\\[3mm]
\mu^{n+\frac{1}{2}}=\mathcal{L}\dfrac{\phi^{n+1}+\phi^{n}}{2}
  + \dfrac{r^{n+1}r^{n}}{\mathcal{E}_{_\mathcal{N}}[\phi^{n+\frac{1}{2}}_{\ast}]+C}\,\mathcal{N}(\phi^{n+\frac{1}{2}}_{\ast}),\\[5mm]
\dfrac{(r^{n+1})^{2}-(r^{n})^{2}}{\tau}= \dfrac{r^{n+1}r^{n}}{\mathcal{E}_{_\mathcal{N}}[\phi^{n+\frac{1}{2}}_{\ast}]+C}
\Bigl(\mathcal{N}(\phi^{n+\frac{1}{2}}_{\ast}),\,\dfrac{\phi^{n+1}-\phi^{n}}{\tau}\Bigr),
\end{array}\right.\label{eqc2}
\end{equation}
where $\phi^{n+\frac{1}{2}}_{\ast}$ is an explicit second-order approximation
of $\phi(t)$ at $t=t^{n+\frac{1}{2}}$.

The product form $r^{n+1}r^{n}$ is the key design choice for the second-order scheme.
These three expressions
\begin{equation}
r^{n+1}r^{n},\qquad
\frac{1}{2}\bigl[(r^{n+1})^{2}+(r^{n})^{2}\bigr],\qquad
\Bigl(\frac{r^{n+1}+r^{n}}{2}\Bigr)^{2},\label{eqthree}
\end{equation}
are consistent with $r(t^{n+1/2})^{2}$ and preserve the
second-order accuracy of the overall Crank--Nicolson discretization.
Substituting either the arithmetic mean or the squared mean
into~\eqref{eqc2}, however, yields a cubic equation in $r^{n+1}$,
destroying the quadratic structure required for the positivity argument.
The product form alone preserves the quadratic character of the equation,
with the same strictly negative constant term $-(r^{n})^{2}$,
and thus guarantees a unique positive root by exactly the same convexity
argument used for the backward Euler scheme.

Eliminating $\mu^{n+\frac{1}{2}}$ and setting $\phi^{n+1}=p^{n+1}+r^{n+1}q^{n+1}$,
we define $p^{n+1}$ and $q^{n+1}$ as the solutions of
\begin{equation}
(\mathcal{I}+\frac{\tau}{2}\mathcal{G}\mathcal{L})\,p^{n+1}
= (\mathcal{I}-\frac{\tau}{2}\mathcal{G}\mathcal{L})\,\phi^{n},\quad
(\mathcal{I}+\frac{\tau}{2}\mathcal{G}\mathcal{L})\,q^{n+1}
= -\frac{\tau\,r^{n}\,\mathcal{G}\mathcal{N}(\phi^{n+\frac{1}{2}}_{\ast})}{\mathcal{E}_{_\mathcal{N}}[\phi^{n+\frac{1}{2}}_{\ast}]+C}.
\label{eqc3}
\end{equation}
Substituting this decomposition into the third equation of~\eqref{eqc2}
yields a quadratic equation for $r^{n+1}$:
\begin{align}
h_{_{\rm CN}}(r^{n+1}):=&
\left[1- \frac{r^{n}(\mathcal{N}(\phi^{n+\frac{1}{2}}_{\ast}),q^{n+1})}{\mathcal{E}_{_\mathcal{N}}[\phi^{n+\frac{1}{2}}_{\ast}]+C}\right]
(r^{n+1})^{2}
- \frac{r^{n}(\mathcal{N}(\phi^{n+\frac{1}{2}}_{\ast}),p^{n+1}-\phi^{n})}{\mathcal{E}_{_\mathcal{N}}[\phi^{n+\frac{1}{2}}_{\ast}]+C}\,r^{n+1}\nonumber\\[1mm]
&-(r^{n})^{2}=0.\label{eqc4}
\end{align}

Under the induction hypothesis $r^{n}>0$, we have $h_{_{\rm CN}}(0)=-(r^{n})^{2}<0$.
Moreover, exactly as in the backward Euler case, the positive definiteness of
$\mathcal{G}^{-1}+\frac{\tau}{2}\mathcal{L}$ guarantees that the coefficient of
$(r^{n+1})^{2}$ in~\eqref{eqc4} is at least~$1$, so $h_{_{\rm CN}}(\cdot)$ is strictly convex.
Together, these properties imply a unique positive root, which can be obtained from the quadratic formula.

\begin{thm}\label{thmf3}
The SAV-CN scheme~\eqref{eqc2} admits a unique solution with $r^{n+1}>0$.
\end{thm}

\begin{thm}\label{thmf4}
The SAV-CN scheme~\eqref{eqc2} unconditionally satisfies the energy dissipation law
\begin{equation}
\frac{\mathcal{E}[\phi^{n+1},r^{n+1}]-\mathcal{E}[\phi^{n},r^{n}]}{\tau}
= -(\mu^{n+\frac{1}{2}},\mathcal{G}\mu^{n+\frac{1}{2}}) \le 0,
\label{eqc6}
\end{equation}
where
\begin{equation}
\mathcal{E}[\phi^{n},r^{n}] = \frac{1}{2}(\phi^{n},\mathcal{L}\phi^{n}) + (r^{n})^{2}-C.
\label{eqc6b}
\end{equation}
\end{thm}

The proofs of Theorems~\ref{thmf3} and~\ref{thmf4} parallel those of Theorems~\ref{thmf1} and~\ref{thmf2}, respectively, and are omitted for brevity.

%Note that the SAV-CN scheme satisfies an \emph{exact} energy equality,
%whereas the SAV-BE scheme yields only an inequality---this reflects the symmetric
%character of the Crank--Nicolson discretization, which eliminates the
%numerical dissipation term $(\mathcal{L}(\phi^{n+1}-\phi^{n}), \phi^{n+1}-\phi^{n})$
%present in the BE proof.

The proposed reformulation eliminates the sign-indeterminacy of the original
SAV method while preserving the square-root definition of~$r$
and the computational cost: two constant-coefficient linear elliptic solves
plus one quadratic formula per time step.
Both schemes can be naturally combined with relaxed~\cite{JiangM22,LiuZ24,Huang24}
or weighted~\cite{Huang26} SAV approaches, which improve the consistency
between $r^{n+1}$ and $\sqrt{\mathcal{E}_{_\mathcal{N}}[\phi^{n+1}]+C}$
through a post-processing step that does not interfere with the
positivity guarantee established here.

%===============================================================
\section{Numerical results}
\label{sec4}

In this section, we present numerical experiments to verify the positivity preservation,
energy stability, and convergence rates of the proposed schemes.
Spatial discretization uses the Fourier spectral method with periodic boundary conditions,
and numerical integration is performed by the composite trapezoidal rule.
The intermediate value $\phi^{n+\frac{1}{2}}_{\ast}$ required in the SAV-CN scheme
is computed via a backward Euler step of size $\frac{\tau}{2}$.

We consider the commonly used Ginzburg--Landau free energy~\cite{ShenJ18,Huang19b,JiangM22}
\begin{equation}
\mathcal{E}[\phi]=\int_{\Omega}\Bigl[\frac{\varepsilon^{2}}{2}|\nabla\phi|^{2}
+\frac{1}{4}(1-\phi^{2})^{2}\Bigr]\mathrm{d}\bm{x},
\label{eqe1}
\end{equation}
where $0<\varepsilon\ll 1$ governs the interfacial width.
Following the SAV decomposition, we set
$\mathcal{L}=-\varepsilon^{2}\Delta+\gamma$ and
$\mathcal{E}_{_\mathcal{N}}[\phi]=\int_{\Omega}F(\phi)\,\mathrm{d}\bm{x}$
with $F(\phi)=\frac{1}{4}(\phi^{2}-1-\gamma)^{2}$, $\gamma\ge 0$.
The $H^{-1}$ gradient flow ($\mathcal{G}=-\Delta$) then yields the Cahn--Hilliard equation.
In the examples below we take $\varepsilon=0.05$, $\gamma=4$ and $C=0$.

Before presenting systematic convergence tests,
we contrast the original and the proposed positivity-preserving SAV-BE scheme on a single
time step of the 2D Cahn--Hilliard equation on
$\Omega=[0,2\pi)^{2}$ with $256^{2}$ Fourier modes.
Take the initial condition
\begin{equation}
\phi(x,y,0) = 4\cos(20x)\sin(20y),\label{eqs1}
\end{equation}
and the time step $\tau=10^{-3}$, for which
$r^{0}= \sqrt{\mathcal{E}_{_\mathcal{N}}[\phi^{0}]+C}\approx 14.3966>0$.
The original SAV-BE scheme~\eqref{eq7} yields
\begin{equation}
r^{1}_{\rm original}\approx -1.6069<0.\label{eqs2}
\end{equation}
The proposed positivity-preserving SAV-BE scheme~\eqref{eqb2} gives $h_{_{\rm BE}}(r^{n+1})=0$ with constant term
$-(r^{0})^{2}<0$, whose unique positive root is
\begin{equation}
r^{1}_{\rm new}\approx 4.0480>0.\label{eqs3}
\end{equation}
Both schemes start from the same $r^{0}>0$; the original produces
a negative auxiliary variable, while the new scheme preserves
positivity, confirming the unconditional guarantee of
Theorem~\ref{thmf1}.

We note that $\tau=10^{-3}$ is already within the typical range
used in phase-field simulations; the sign failure persists because
in Fourier space $\mathcal{G}\mathcal{L}$ acts as multiplication by
$\lambda(\bm{k})=|\bm{k}|^{2}(\varepsilon^{2}|\bm{k}|^{2}+\gamma)$;
for $\cos(20x)\sin(20y)$ the wave vector is
$\bm{k}=(\pm20,\pm20)$, hence $|\bm{k}|^{2}=20^{2}+20^{2}=800$,
giving $\lambda=800(\varepsilon^{2}\cdot800+\gamma)=4800$,
so that $\tau\lambda/(1+\tau\lambda)\approx 0.8276$ even at
$\tau=10^{-3}$, illustrating that the problem is structural
rather than a large-$\tau$ artefact.
A practical consequence is that, for a given spatial
resolution, the time step to guarantee positivity may be far smaller than what is practical
for phase-field simulations; high-frequency modes keep the factor
$\tau\lambda/(1+\tau\lambda)$ close to unity even when $\tau$ is
reduced to values within the typical range; enforcing positivity
through $\tau$ alone would require time steps so small as to be
computationally prohibitive.
The new scheme removes this concern entirely by making the positivity of $r^{n+1}$ unconditional.

In the following, the initial condition is prescribed as
\begin{equation}
\phi(x, y, 0) = \tanh\left(\frac{\operatorname{dist}((x,y), \Gamma_{0})}{\sqrt{2\varepsilon}}\right),\label{eqee2}
\end{equation}
where $\Gamma_{0}$ is the closed parametric curve~\cite{Huang26}
\begin{equation}
\Gamma_{0} = \left\{(x,y)\mid x = \cos\theta,\; y = 2\sin\theta - 1.9\sin^3\theta,\; \theta\in[0,2\pi)\right\},
\end{equation}
and $\operatorname{dist}(\cdot, \Gamma_{0})$ denotes the signed distance to $\Gamma_{0}$.
The computational domain $\Omega = [0,3)^2$ is discretized using $256^2$ Fourier modes.

Table~\ref{tab1} lists the $l^{2}$- and $l^{\infty}$-errors and convergence rates
of the proposed positivity-preserving SAV-BE and SAV-CN schemes at $T=0.5$ under successively refined time steps.
The SAV-BE scheme exhibits first-order accuracy while the SAV-CN scheme
achieves second-order accuracy, consistent with theoretical expectations.

\begin{table}[!htp]
\renewcommand\arraystretch{1.2}
\centering
\caption{Temporal errors and convergence rates of the positivity-preserving SAV-BE and SAV-CN schemes at $T=0.5$.}\label{tab1}
\tabcolsep 0.02in
\resizebox{\textwidth}{!}{%
\begin{tabular}{@{}c c c c c c c c c@{}}
\toprule
\multirow{2}{*}{$\tau$} &
\multicolumn{4}{c}{SAV-BE} &
\multicolumn{4}{c}{SAV-CN} \\
\cmidrule(lr){2-5} \cmidrule(lr){6-9}
& $\|\phi_{\tau} - \phi_{\tau/2}\|_{l^{2}}$ & Rate &
  $\|\phi_{\tau} - \phi_{\tau/2}\|_{l^{\infty}}$ & Rate &
  $\|\phi_{\tau} - \phi_{\tau/2}\|_{l^{2}}$ & Rate &
  $\|\phi_{\tau} - \phi_{\tau/2}\|_{l^{\infty}}$ & Rate \\
\midrule
$0.5^{0} \times 10^{-4}$
    & 7.5869e-03 & --    & 1.5912e-02 & --
    & 2.4317e-04 & --    & 5.2598e-04 & -- \\
$0.5^{1} \times 10^{-4}$
    & 3.8345e-03 & 0.9845 & 8.0269e-03 & 0.9872
    & 7.0661e-05 & 1.7830 & 1.5352e-04 & 1.7766 \\
$0.5^{2} \times 10^{-4}$
    & 1.9277e-03 & 0.9921 & 4.0315e-03 & 0.9935
    & 1.9559e-05 & 1.8531 & 4.2739e-05 & 1.8448 \\
$0.5^{3} \times 10^{-4}$
    & 9.6654e-04 & 0.9960 & 2.0208e-03 & 0.9964
    & 5.2371e-06 & 1.9010 & 1.1532e-05 & 1.8899 \\
$0.5^{4} \times 10^{-4}$
    & 4.8396e-04 & 0.9979 & 1.0117e-03 & 0.9981
    & 1.3772e-06 & 1.9270 & 3.0655e-06 & 1.9115 \\
$0.5^{5} \times 10^{-4}$
    & 2.4217e-04 & 0.9989 & 5.0623e-04 & 0.9989
    & 3.6087e-07 & 1.9322 & 8.1522e-07 & 1.9109 \\
\bottomrule
\end{tabular}}%
\end{table}

\begin{figure}[htp]
\centering
\includegraphics[width=12cm]{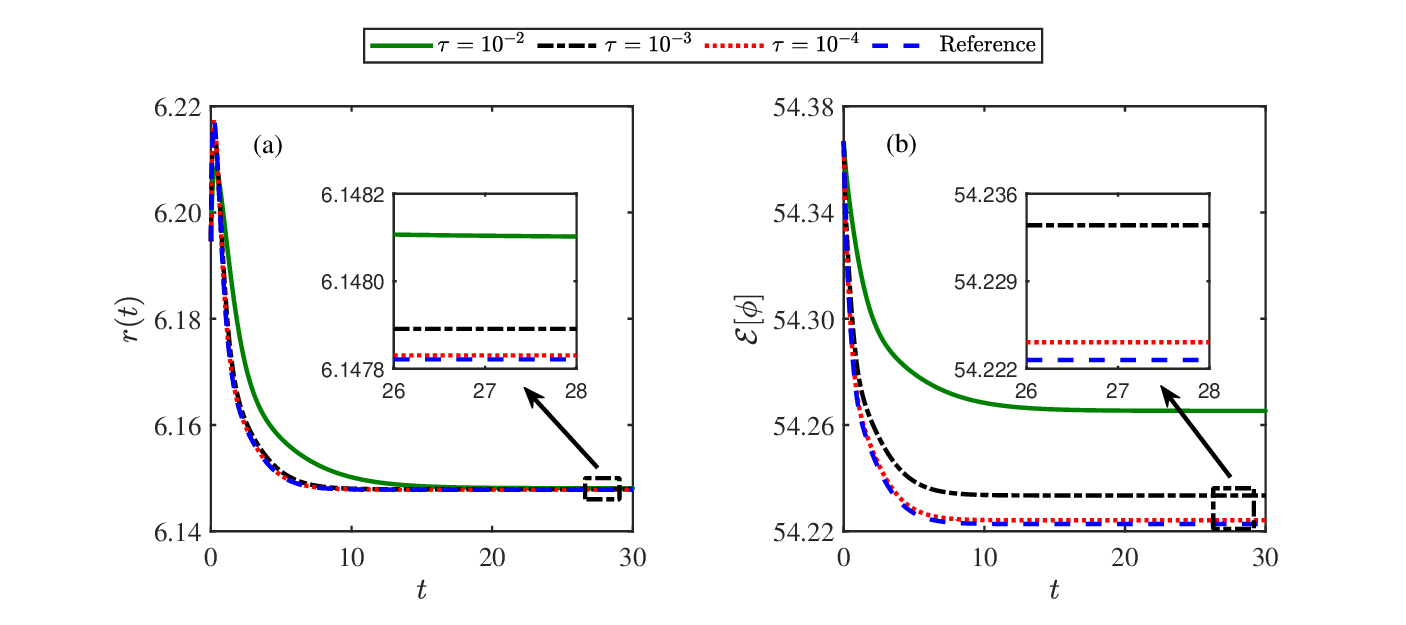}
\caption{Time evolution of the (a) discrete auxiliary variable $r^{n}$ and (b) discrete energy
         $\mathcal{E}[\phi^{n},r^{n}]$ for the positivity-preserving SAV-BE scheme~\eqref{eqb2}.
         The reference solution is obtained using a second-order Lagrange multiplier
         scheme~\cite{ChengQ20} with $\tau=10^{-5}$.}\label{fig1}
\end{figure}

\begin{figure}[htp]
\centering
\includegraphics[width=12cm]{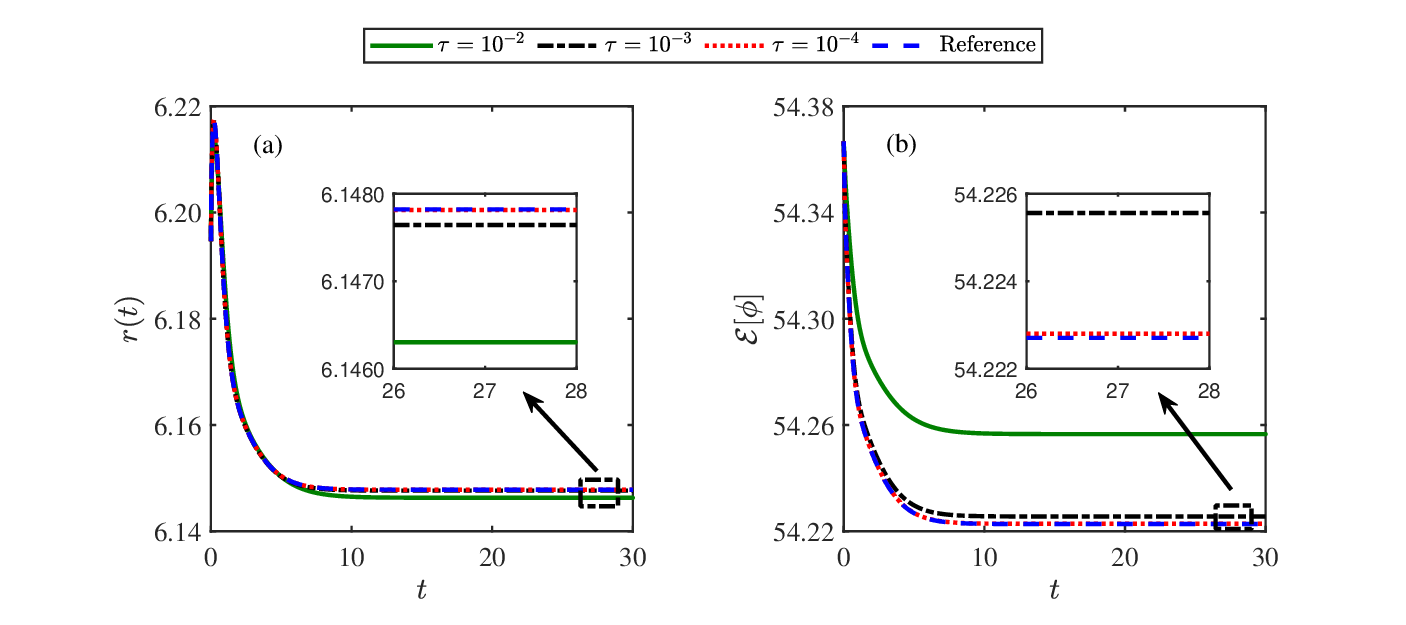}
\caption{Time evolution of the (a) discrete auxiliary variable $r^{n}$ and (b) discrete energy
         $\mathcal{E}[\phi^{n},r^{n}]$ for the positivity-preserving SAV-CN scheme~\eqref{eqc2}.
         The reference solution is obtained using a second-order Lagrange multiplier
         scheme~\cite{ChengQ20} with $\tau=10^{-5}$.}\label{fig2}
\end{figure}

Figures~\ref{fig1} and~\ref{fig2} display the evolution of the auxiliary
variable $r^{n}$ and the discrete energy $\mathcal{E}[\phi^{n},r^{n}]$
under the positivity-preserving SAV-BE and SAV-CN schemes, respectively.
Two observations are noteworthy.
First, for all time-step sizes tested, $r^{n}$ remains strictly positive
throughout the entire simulation, confirming the unconditional positivity
guarantee of Theorems~\ref{thmf1} and~\ref{thmf3}.
Second, the discrete energy decays monotonically in every case,
verifying the unconditional energy stability proved in
Theorems~\ref{thmf2} and~\ref{thmf4}.
As $\tau$ decreases, both $r^{n}$ and $\mathcal{E}[\phi^{n},r^{n}]$ converge
to the reference solution.
At comparable step sizes, the SAV-CN scheme yields energy and auxiliary
variable profiles substantially closer to the reference than the SAV-BE
scheme, consistent with its second-order temporal accuracy.

\section{Conclusion}
\label{sec5}

This work has resolved a long-standing but overlooked issue of
sign-indeterminacy of the auxiliary variable in the original SAV method
while preserving its square-root definition.
We first gave a quantitative analysis of the original scheme: a sharp
criterion and a sufficient condition for $r^{n+1}>0$ were derived, and an
explicit counterexample demonstrated that sign failure is structural
rather than a numerical pathology; a single-step PDE comparison
further illustrated the contrast between the original and the new
scheme on a Cahn--Hilliard problem.
The central idea is to reformulate the discrete evolution from
$r_t$ to $(r^{2})_t$, which transforms the scalar auxiliary equation
from a linear equation with an indefinite constant term into a
convex quadratic whose constant term $-(r^{n})^{2}$ is strictly negative,
thereby always admitting a unique positive root obtained from the
quadratic formula.
For the Crank--Nicolson scheme, the product-form discretization
$r^{n+1}r^{n}$ is essential: it preserves the quadratic structure,
whereas the arithmetic or squared means lead to cubic equations
that admit no comparable guarantee.
The resulting SAV-BE and SAV-CN schemes incur exactly the same
computational cost as the original SAV method and are rigorously
proved to be unconditionally energy-stable.

Several directions merit further investigation.
First, the proposed reformulation can be combined with
relaxed~\cite{JiangM22,LiuZ24,Huang24}
or weighted~\cite{Huang26} SAV frameworks to simultaneously improve
consistency between $r^{n+1}$ and $\sqrt{\mathcal{E}_{_\mathcal{N}}[\phi^{n+1}]+C}$
while retaining the unconditional positivity guarantee.
Second, higher-order extensions via diagonally implicit
Runge--Kutta (DIRK) methods are natural: each DIRK stage reduces to
a BE-type quadratic, so positivity carries over directly.
In contrast, any linear multi-step method beyond the one-step case,
including BDF-$k$ ($k\ge 2$) and the general multi-step schemes
studied in~\cite{Huang23},
introduces more than two time levels into the $(r^{2})_{t}$
discretisation, which breaks the quadratic structure of the
constant term and precludes an unconditional positivity proof.
Third, rigorous error estimates analogous to those for the
original SAV method~\cite{ShenJ18b,ChenH20}
would complete the theoretical foundation.
Finally, although the present analysis focuses on gradient flows,
the $(r^{2})_{t}$ reformulation is purely algebraic in nature
and may also apply to SAV discretisations of Hamiltonian
and other energy-conserving systems.

%---------------------------------------------------------------------
\begin{acknowledgement}
The numerical calculations in this paper have been done on the supercomputing system in the Supercomputing Center of Wuhan University.
\end{acknowledgement}
\medskip
\noindent{\textbf{Funding}} 
~This work was partially supported by the National Natural Science Foundation of China (Nos.~12001210, 12131010), 
the Natural Science Foundation of Henan Province (No. 252300420308) 
and the Key Scientific Research Project of Universities in Henan Province (No. 27A110002).

\medskip
\noindent{\textbf{Data Availability}} ~The code for the current study is available from the corresponding author on reasonable request.

\medskip
\noindent{\textbf{Declarations}}\\
\noindent{\textbf{Conflict of interest}} ~The authors declare no competing interests.

%---------------------------------------------------------------------

\end{document}